\documentclass[preprint,12pt]{elsarticle}

\usepackage{amssymb}
\usepackage{amsmath}
\usepackage{amsfonts}
\usepackage{graphicx}
\usepackage{graphics}
\usepackage{tikz}
\usepackage{titlesec}

\newtheorem{theorem}{Theorem}

\newtheorem{corollary}[theorem]{Corollary}

\newtheorem{definition}[theorem]{Definition}
\newtheorem{example}[theorem]{Example}
\newtheorem{remark}[theorem]{Remark}

\newtheorem{lemma}[theorem]{Lemma}

\newtheorem{proposition}[theorem]{Proposition}

\def\qed{\vbox{\hrule
 \hbox{\vrule\hbox to 5pt{\vbox to 8pt{\vfil}\hfil}\vrule}\hrule}}

\usepackage{tikz}
\usepackage{tkz-graph}

\journal{Linear Algebra and Its Applications}

\begin{document}
\begin{frontmatter}

\title{Realizable lists via the spectra of structured
matrices}

\author{Cristina Manzaneda}
\address{Departamento de Matem\'{a}ticas, Facultad de Ciencias. Universidad Cat\'{o}lica del Norte. Av. Angamos 0610 Antofagasta, Chile.}
\ead{cmanzaneda@ucn.cl}

\author[]{Enide Andrade\corref{cor1}}
\address{CIDMA-Center for Research and Development in Mathematics and Applications
         Departamento de Matem\'atica, Universidade de Aveiro, 3810-193, Aveiro, Portugal.}
\ead{enide@ua.pt}
\cortext[cor1]{Corresponding author}

\author{Mar\'{\i}a Robbiano}
\address{Departamento de Matem\'{a}ticas, Facultad de Ciencias. Universidad Cat\'{o}lica del Norte. Av. Angamos 0610 Antofagasta, Chile.}
\ead{mrobbiano@ucn.cl}

\begin{abstract}

A square matrix of order $n$ with $n\geq 2$ is called a \textit{permutative matrix} or permutative when all its rows (up to the first one) are permutations of precisely its first row. In this paper, the spectra of a class of permutative matrices are studied.
In particular, spectral results for matrices partitioned into $2$-by-$2$ symmetric blocks are presented and, using these results sufficient conditions on a given list to be the list of eigenvalues of a nonnegative permutative matrix are obtained and the corresponding permutative matrices are constructed. Guo perturbations on given lists are exhibited.

\end{abstract}

\begin{keyword}
permutative matrix; symmetric matrix; inverse eigenvalue problem; nonnegative matrix.

\MSC 15A18, 15A29, 15B99.

\end{keyword}

\end{frontmatter}

\section{Introduction}
We present here a short overview related with the nonnegative inverse eigenvalue problem (NIEP) that is the problem of determining necessary and sufficient conditions for a list of complex
numbers

\begin{equation}
\sigma =\left( \lambda _{1},\lambda _{2},\ldots ,\lambda _{n}\right)
\label{list}
\end{equation}
to be the spectrum of a $n$-by-$n$ entrywise nonnegative matrix $A$. If a list $\sigma$ is the
spectrum of a nonnegative matrix $A$, then $\sigma$ is \textit{realizable} and the matrix $A$ \textit{realizes} $\sigma,$ (or, that is a realizing matrix for the list). This problem attracted the attention of many authors over $50+$ years and it was firstly considered by Sule\u{\i}manova \cite{SLMNva} in $1949$. Although some partial results were obtained the NIEP is an open problem for $n \geq 5$.  In \cite{LwyLdn} this problem was solved for $n=3$ and for matrices of order $n=4$ the problem was solved  in \cite{Meehan} and \cite{Mayo}. It has been studied in its general form in e.g. \cite{Boyle,Johnson,Laffey2,muitos,LwyLdn,SmgcH,NN3,GuoWen}. When the realizing nonnegative matrix is required to be symmetric (with, of course, real eigenvalues) the problem is designated by symmetric nonnegative inverse eigenvalue problem (SNIEP) and it is also an open problem. It has also been the subject of considerable attention e.g \cite{Fiedler,Laffey,LMc,Soules}. The problem of which lists of $n$ real numbers can occur as eigenvalues of an $n$-by-$n$  nonnegative matrix is called real nonnegative inverse eigenvalue problem (RNIEP), and some results can be seen in e.g. \cite{Boro,Friedland,HzlP, RS, SRM}.
In what follows $\sigma\left(A\right)$  denotes the set of eigenvalues of a square matrix $A$.
Below are listed some necessary conditions on a list of complex numbers $\sigma =\left( \lambda _{1},\lambda _{2},\ldots ,\lambda _{n}\right)$ to be the spectrum of a nonnegative matrix.

\begin{enumerate}
\item The Perron eigenvalue $\max \left\{ \left\vert \lambda \right\vert
:\lambda \in \sigma(A) \right\} $ belongs to $\sigma.$

\item The list $\sigma $ is closed under complex conjugation.

\item $s_{k}\left( \sigma \right) =\sum\limits_{i=1}^{n}\lambda
_{i}^{k}\geq 0.$

\item $s_{k}^{m}\left( \sigma \right) \leq n^{m-1}s_{km}\left( \sigma
\right)$ for $ k,m=1,2,\ldots $.
\end{enumerate}
The first condition listed above follows from the Perron-Frobenius theorem,
which is an important theorem in the theory of nonnegative matrices. The last condition was
proved by Johnson \cite{Johnson} and independently by Loewy and London \cite{LwyLdn}. The necessary
conditions that were presented for the NIEP are sufficient only when the list $\sigma$ has at
most three elements. The solution for NIEP was also found for lists with
four elements, while the problem for lists with five or more elements is
still open.

\begin{definition}
{\rm The list $\sigma$ in (\ref{list}) is a Sule\u{\i}manova spectrum if the $
\lambda ^{\prime }s$ are real numbers, $\lambda _{1}>0\geq \lambda _{2}\geq
\cdots \geq \lambda _{n}$ and $s_{1}\left( \sigma \right) \geq 0$.}
\end{definition}

Sule\u{\i}manova, \cite{SLMNva} stated (and loosely proved) that every such spectrum is
realizable. Fiedler \cite{Fiedler} proved that every Sule\u{\i}manova
spectrum is symmetrically realizable (i.e. realizable by a symmetric nonnegative matrix).

One of the most promising attempts to solve the NIEP is to identify the
spectra of certain structured matrices with known characteristic polynomials. Friedland in \cite{Friedland} and Perfect in \cite{Perfect} proved
Sule\u{\i}manova's result via companion matrices of certain polynomials.  However, constructing the companion matrix of a Sule\u{\i}manova's spectrum is computationally difficult. Recently, Paparella  \cite{PP} gave a constructive proof of Sule\u{\i}manova's result. The author defined \textit{permutative} matrix as follows.

\begin{definition} {\rm \cite{PP} \label{def2}
Let $\mathbf{x=}\left( x_{1},\ldots ,x_{n}\right) ^{T}\in \mathbb{C}^{n}$.\
Let $P_{2},\ldots ,P_{n}$ be permutation matrices.\ A \textit{permutative matrix}\
$P$ is a\ matrix which takes the form
\begin{equation*}
P=
\begin{pmatrix}
\mathbf{x}^{T} \\
\left( P_{2}\mathbf{x}\right) ^{T} \\
\vdots  \\
\left( P_{n-1}\mathbf{x}\right) ^{T} \\
\left( P_{n}\mathbf{x}\right) ^{T}%
\end{pmatrix}
.
\end{equation*}
}
\end{definition}

\noindent In \cite{PP}, explicit permutative matrices which realize Sule\u{\i}manova
spectra were found.
A few remarks concerning the brief history of permutative matrices are in
order.
\begin{enumerate}
\item Ranks of permutative matrices were studied by Hu et al. \cite{Hu}
\item Moreover, the author \cite{PP} proposed the interesting problem which asks
if all realizable spectra can be realizable by a permutative matrix or by a direct sum of
permutative matrices. An equivalent problem communicated to the author by R. Loewy is to find an extreme
nonnegative matrix \cite{Laffey2} with real spectrum that can not be realized by a
permutative matrix or a direct sum of permutative matrices. Loewy \cite{Lwy} resolved this problem in the negative by showing that the list $ \sigma= \left( 1,\frac{8}{25}+\frac{\sqrt{51}}{50},\frac{8}{25}+\frac{\sqrt{51}}{50},-\frac{4}{5}, -\frac{21}{25}\right) $ is realizable but cannot be realized by a permutative matrix or by a direct sum of permutative matrices.
\end{enumerate}

In this paper we call the problem as PNIEP when the NIEP involves permutative matrices.
Note that the lists considered along the paper are equivalent (up to a permutation of its elements).
Therefore, unless we say the contrary, we call a given $n$-tuple $\sigma$ or any permutation resulting from it, as ``the list". In consequence, any of these lists can be used.

In this work we will find spectral results for partitioned into $2$-by-$2$ blocks matrices and using these results sufficient conditions on given lists to be the list of eigenvalues of a nonnegative permutative matrix are obtained. The paper is organized as follows: At Section 2 some definitions and facts related to permutative matrices are given. At Section 3 spectral results for matrices partitioned into $2$-by-$2$ blocks are presented and the results are applied to NIEP, SNIEP and PNIEP. Some illustrative examples are provided. At Section 4 results for matrices with odd order are presented. Finally, at Section 5  Guo perturbations on lists of eigenvalues of this class of permutative matrices (in order to obtain a new permutative matrix) are studied.

\section{Permutatively equivalent matrices}
In this section some auxiliary results from \cite{PP} and some new definitions are introduced.
In \cite{PP} the following results were proven.

\begin{lemma}
\label{pappa copy(1)}\ \cite[Lemma 3.1]{PP} For $\mathbf{x}=\left(
x_{1},x_{2},\ldots,x_{n}\right)  ^{T}\in\mathbb{C}^{n}$, let
\begin{equation}
X=
\begin{pmatrix}
x_{1} & x_{2} & \ldots & x_{i} & \ldots & x_{n-1} & x_{n}\\
x_{2} & x_{1} & \ldots & x_{i} & \ldots & x_{n-1} & x_{n}\\
\vdots & \vdots & \ddots & \vdots & \ddots & \vdots & \vdots\\
x_{i} & x_{2} & \ddots & x_{1} & \ddots & \vdots & \vdots\\
\vdots & \vdots & \vdots & \vdots & \ddots & \vdots & \vdots\\
x_{n-1} & x_{2} & \ldots & \vdots & \vdots & x_{1} & x_{n}\\
x_{n} & x_{2} & \ldots & x_{i} & \ldots & x_{n-1} & x_{1}%
\end{pmatrix}
.\label{mX}%
\end{equation}
Then, the set of eigenvalues of $X$ is given by
\begin{equation}
\sigma(X)=\left\{
{\displaystyle\sum\limits_{i=1}^{n}}
x_{i},x_{1}-x_{2},x_{1}-x_{3},\ldots,x_{1}-x_{n}\right\}.
\end{equation}
\end{lemma}

\begin{theorem} \label{pappa} \cite{PP} Let $\sigma=\left( \lambda_{1},\ldots,\lambda
_{n}\right)  $ be a Sule\u{\i}manova spectrum and consider the $n$-tuple
$\mathbf{x=}\left(  x_{1},x_{2},\ldots,x_{n}\right)  $, where
\[%
\begin{tabular}
[c]{ccc}%
$x_{1}=\frac{\lambda_{1}+\cdots+\lambda_{n}}{n}$ & and & $x_{i}=x_{1}%
-\lambda_{i}, \, 2\leq i\leq n,$
\end{tabular}
\
\]
then the matrix in (\ref{mX}) realizes $\sigma$.\ In particular, if
$\lambda_{1}+\cdots+\lambda_{n}=0$ the solution matrix, $X_{0}$ becomes
\[
X_{0}=
\begin{pmatrix}
0 & \left\vert \lambda_{2}\right\vert  & \ldots & \left\vert \lambda
_{i}\right\vert  & \ldots & \left\vert \lambda_{n-1}\right\vert  & \left\vert
\lambda_{n}\right\vert \\
\left\vert \lambda_{2}\right\vert  & 0 & \ldots & \left\vert \lambda
_{i}\right\vert  & \ldots & \left\vert \lambda_{n-1}\right\vert  & \left\vert
\lambda_{n}\right\vert \\
\vdots & \vdots & \ddots & \vdots & \ddots & \vdots & \vdots\\
\left\vert \lambda_{i}\right\vert  & \left\vert \lambda_{2}\right\vert  &
\ddots & 0 & \ddots & \vdots & \vdots\\
\vdots & \vdots & \vdots & \vdots & \ddots & \vdots & \vdots\\
\left\vert \lambda_{n-1}\right\vert & \left\vert \lambda_{2}\right\vert  & \vdots & \vdots & \vdots & 0 &
\left\vert \lambda_{n}\right\vert \\
\left\vert \lambda_{n}\right\vert  & \left\vert \lambda_{2}\right\vert  &
\ldots & \left\vert \lambda_{i}\right\vert  & \ldots & \left\vert \lambda
_{n-1}\right\vert  & 0
\end{pmatrix}.
\]
\end{theorem}

\begin{remark} {\rm
\label{newR} By previous results and the proof of above Theorem \ref{pappa} in \cite{PP}, it is clear that for any set $\sigma=\left\{\alpha_{1},\ldots,\alpha_{n}\right\}$ there exists a permutative matrix with
the shape of $X$ in (\ref{mX}) whose set of eigenvalues is $\sigma.$}
\end{remark}

\noindent The following notions will be used in the sequel.

\begin{definition}
{\rm
\label{ept}
Let $\mathbf{\tau }=\left( \tau _{1},\ldots ,\tau _{n}\right) $ be an $n$-tuple whose elements are permutations in the symmetric group\ $S_{n}$, with $\tau _{1}=id$.\ Let $\mathbf{a=}\left( a_{1},\ldots ,a_{n}\right) \in
\mathbb{C}^{n}$. Define the row-vector,
\begin{equation*}
\tau _{j}\left( \mathbf{a}\right) =\left( a_{\tau _{j}\left( 1\right)
},\ldots ,a_{\tau _{j}\left( n\right) }\right)
\end{equation*}%
and consider the matrix
\begin{equation}
\tau \left( \mathbf{a}\right) =%
\begin{pmatrix}
\tau _{1}\left( \mathbf{a}\right)   \\
\tau _{2}\left( \mathbf{a}\right)  \\
\vdots \\
\tau _{n-1}\left( \mathbf{a}\right)  \\
\tau _{n}\left( \mathbf{a}\right)
\end{pmatrix}
.  \label{permut}
\end{equation}
An $n$-by-$n$ matrix $A$, is called $\mathbf{\tau}$\emph{-permutative} if
$A=\tau \left( \mathbf{a}\right) $ for some $n$-tuple $\mathbf{a}$.}
\end{definition}

\begin{remark} {\rm Although the statement in Definition \ref{def2} is precisely the statement found in  \cite[Definiton 2.1]{PP}, it is clear that Definition \ref{ept} of this work is the proper definition of a permutative matrix (indeed,
since every permutation matrix is a permutative matrix, it is not ideal to define the latter with the former). Thus, Definition \ref{ept} is a better definition of a permutative matrix than the one given at Definition \ref{def2}.
}
\end{remark}

\begin{definition}
{\rm
If $A$ and $B$ are $\mathbf{\tau }$-permutative by a common vector
$\mathbf{\tau }=\left( \tau _{1},\ldots ,\tau _{n}\right)$
then they are called \textit{permutatively equivalent}.}
\end{definition}

\begin{definition}
{\rm
Let $\varphi \in S_{n}$ and the $n$-tuple $\mathbf{\tau }=\left(
id,\varphi ,\varphi ^{2},\ldots ,\varphi ^{n-1}\right) \in \left(
S_{n}\right) ^{n}.$ Then a $\mathbf{\tau }$-permutative matrix is called $
\varphi $\textit{-permutative}.}
\end{definition}

It is clear from the definitions that two
$\varphi $-permutative matrices are permutatively equivalent matrices.

\begin{remark}
{\rm
If permutations are regarded as bijective
maps from the set $\left\{ 0,1,\cdots, n-1\right\}$ to itself, then a circulant (respectively, left circulant) matrix
is a $\varphi $-permutative matrix
where
$\varphi \left( i \right)\equiv i-1(\mbox{mod}\ {n})$ (resp. $\varphi \left( i \right)\equiv i+1(\mbox{mod}\ {n})$).
Indeed, notice that the $3$-by-$3$ circulant matrix
\begin{equation*}
\begin{pmatrix}
a & b & c \\
c & a & b \\
b & c & a
\end{pmatrix}
\end{equation*}
is $\varphi$- permutative with
$$\varphi =
\left ( \begin{array}{ccc}
    0 & 1 & 2  \\
    2 & 0 & 1
\end{array}
\right)
$$
and the $3$-by-$3$ left circulant matrix
\begin{equation*}
\begin{pmatrix}
a & b & c \\
b & c & a \\
c & a & b
\end{pmatrix}
\end{equation*} is $\varphi$- permutative with
$$\varphi =
\left ( \begin{array}{ccc}
    0 & 1 & 2  \\
    1 & 2 & 0
\end{array}
\right).
$$
}
\end{remark}

\begin{remark}
{\rm
\label{important2} A permutative matrix $A$ defines the class of permutatively equivalent matrices. Let $\sigma_{1},\sigma_{2}$ be two Sule\u{\i}manova
spectra, then the corresponding realizing matrices $X_{\sigma_{1}}$ and $X_{\sigma_{2}}$
given by Theorem \ref{pappa} are permutatively equivalent matrices. Furthermore,  by Lemma \ref{pappa copy(1)} and Remark \ref{newR} it is easy to check that given two arbitrary inverse eigenvalue problems (not necessarily NIEP) there exist a solution which is permutatively equivalent to the matrix $X$ in (\ref{mX}).}
\end{remark}

For $\mathbf{\tau}$-permutative matrices an analogous property related with circulant
matrices is given below.

\begin{proposition}
\label{cristi}
Let $\left\{ A_{i}\right\} _{i=1}^{k}$ be a family of permutatively equivalent matrices in $\mathbb{C}^{n\times n}$.
Let $\left\{ \gamma _{i}\right\} _{i=1}^{k}$ be a set of complex numbers. Consider
\begin{equation}
A=\sum\limits_{i=1}^{k}\gamma _{i}A_{i}.
\end{equation}
Then $A_{1}$ and $A$ are permutatively equivalent matrices.
\end{proposition}

\begin{proof}
Let $\mathbf{\tau }=\left( \tau _{1},\ldots ,\tau _{n}\right) $ be an $n$-tuple whose elements are permutations in the symmetric group $S_{n}$
and suppose that the family $\left \{ A_{i}\right\}$ are permutatively equivalent by $\tau$.
Let $\mathbf{e}_{1},\mathbf{e}_{2},\ldots ,\mathbf{e}_{n}$ be the canonical
row vectors in $\mathbb{C}^{n}$. The result is an immediate consequence of the
fact that for any $\mathbf{a=}\left( a_{1},\ldots ,a_{n}\right)\in \mathbb{C}
^{n}$ the matrix $\tau \left( \mathbf{a}\right)$ in (\ref{permut}) can be
decomposed as
\begin{equation*}
\tau \left( \mathbf{a}\right) =a_{1}\tau \left( \mathbf{e}_{1}\right)
+a_{2}\tau \left( \mathbf{e}_{2}\right) +\cdots +a_{n}\tau \left( \mathbf{e}%
_{n}\right),
\end{equation*}
where
\begin{equation*}
\tau \left( \mathbf{e}_{j}\right) =%
\begin{pmatrix}
\mathbf{e}_{j} \\
\tau _{2}\left( \mathbf{e}_{j}\right)   \\
\vdots \\
\tau _{n-1}\left( \mathbf{e}_{j}\right)   \\
\tau _{n}\left( \mathbf{e}_{j}\right) %
\end{pmatrix}.
\end{equation*}
\end{proof}

\section{Eigenpairs for some into block matrices}

In this section we exhibit spectral results for matrices that are partitioned into $2$-by-$2$
symmetric blocks and we apply the results to NIEP, SNIEP and PNIEP. The next theorem is valid in an algebraic closed field $K$ of characteristic $0$. For instance, $K=\mathbb{C}$.

\begin{theorem}
\label{main}
Let $K$ be an algebraically closed field of characteristic $0$ and suppose
that $A=\left( A_{ij}\right) $ is a block matrix of order $2n$, where
\begin{equation}
A_{ij}=%
\begin{pmatrix}
a_{ij} & b_{ij} \\
b_{ij} & a_{ij}
\end{pmatrix}
\text{,\ }a_{ij}\text{,\ }b_{ij}\in K.
\end{equation}
If
\[
s_{ij}=a_{ij}+b_{ij},\ 1\leq i,j\leq n
\]%
and%
\[
c_{ij}=a_{ij}-b_{ij},\ 1\leq i,j\leq n
\]%
Then
\[
\sigma \left( A\right) =\sigma \left( S\right) \cup \sigma \left( C\right)
\]%
where
\[
S=\left( s_{ij}\right) \text{ and }C=\left( c_{ij}\right) .
\]
\end{theorem}

\begin{proof}
Let $\left( \lambda ,v\right) $ be an eigenpair of $S$,\ with $v=\left(
v_{1},\ldots ,v_{n}\right) ^{T}$,  and consider the $2n$-by-$1$ block vector
$w=\left( w_{j}\right) $, where $w_{j}:=v_{j}\mathbf{e}$ and $\mathbf{e=}%
\left( 1,1\right) ^{T}$. Since
\[
A_{ij}w_{j}=%
\begin{pmatrix}
a_{ij} & b_{ij} \\
b_{ij} & a_{ij}%
\end{pmatrix}%
v_{j}\mathbf{e=}%
\begin{pmatrix}
a_{ij}+b_{ij} \\
b_{ij}+a_{ij}%
\end{pmatrix}%
v_{j}=s_{ij}v_{j}\mathbf{e}\text{,}
\]%
notice that, for every $i=1,\ldots ,n,$%
\[
\sum\limits_{j=1}^{n}A_{ij}w_{j}=\left(
\sum\limits_{j=1}^{n}s_{ij}v_{j}\right) \mathbf{e=}\left( \lambda
v_{i}\right) \mathbf{e=}\lambda \left( v_{i}\mathbf{e}\right) =\lambda w_{i}
\]%
i.e $\left( \lambda ,w\right) $ be an eigenpair of $A$.\ Thus $\sigma \left(
S\right) \subseteq \sigma \left( A\right) $.

Similarly, let $\left( \mu ,x\right) $ be an eigenpair of $C,$ with $%
x=\left( x_{1},\ldots ,x_{n}\right) ^{T}$ and consider the $2n$-by-$1$ block
vector $y=\left( y_{j}\right) $, where $y_{j}:=x_{j}\mathbf{f}$ and $\mathbf{%
f=}\left( 1,-1\right) ^{T}$. Since
\[
A_{ij}y_{j}=%
\begin{pmatrix}
a_{ij} & b_{ij} \\
b_{ij} & a_{ij}%
\end{pmatrix}%
x_{j}\mathbf{f=}%
\begin{pmatrix}
a_{ij}-b_{ij} \\
b_{ij}-a_{ij}%
\end{pmatrix}%
x_{j}=c_{ij}x_{j}\mathbf{f}\text{,}
\]%
notice that, for every $i=1,\ldots ,n,$%
\[
\sum\limits_{j=1}^{n}A_{ij}y_{j}=\left(
\sum\limits_{j=1}^{n}c_{ij}x_{j}\right) \mathbf{f=}\left( \lambda
x_{i}\right) \mathbf{f=}\lambda \left( x_{i}\mathbf{f}\right) =\lambda y_{i}
\]%
i.e $\left( \mu ,y\right) $ is also an eigenpair of $A$.\ Thus $\sigma \left(
C\right) \subseteq \sigma \left( A\right) $.
Suppose that
$$\Theta _{s}=\left\{ \left( x_{1i},x_{2i},\ldots ,x_{ni}\right) ^{T}:i=1,\ldots ,n\right\} $$
and
$$\Theta _{c}=\left\{ \left( y_{1i},y_{2i},\ldots , y_{ni}\right) ^{T}:i=1,\ldots ,n\right\}$$
are bases formed with eigenvectors of $S$ and $C$, respectively.
The result will follow after proving the linear independence of the set $\Upsilon= \Upsilon_1 \cup \Upsilon_2$, where:
\begin{equation*}
\Upsilon_1 =\left\{ \left( x_{1}\mathbf{e}_{2}^{T}\mathbf{,}x_{2}\mathbf{e}
_{2}^{T}\mathbf{,\ldots ,}x_{n}\mathbf{e}_{2}^{T}\right) ^{T}:\left( x_{1}
\mathbf{,}x_{2}\mathbf{,\ldots ,}x_{n}\right) ^{T}\in \text{ }\Theta
_{s}\right\}
\end{equation*}
and
\begin{equation*}
\Upsilon_2 =\left\{ \left( y_{1}\mathbf{f}_{2}^{T}\mathbf{,}y_{2}
\mathbf{f}_{2}^{T}\mathbf{,\ldots ,}y_{n}\mathbf{f}_{2}^{T}\right)
^{T}:\left(y_{1}\mathbf{,}y_{2}\mathbf{,\ldots ,}y_{n}\right)^{T}
\text{ }\in \text{ }\Theta _{c}\right\}.
\end{equation*}
Therefore, we consider the following determinant,
\begin{equation*}
d=%
\begin{vmatrix}
y_{11} & \ldots & y_{1n} & x_{11} & \ldots & x_{1n} \\
-y_{11} & \ldots & -y_{1n} & x_{11} & \ldots & x_{1n} \\
\vdots & \ddots & \vdots & \vdots & \ddots & \vdots \\
\vdots & \ddots & \vdots & \vdots & \ddots & \vdots \\
y_{n1} & \ldots & y_{nn} & x_{n1} & \ldots & x_{nn} \\
-y_{n1} & \ldots & -y_{nn} & x_{n1} & \ldots & x_{nn}%
\end{vmatrix}
.
\end{equation*}
Note that $d$ stands for the determinant of a $2n$-by-$2n$ matrix obtained from the coordinates of the vectors
in $\Upsilon.$ By adding rows and after making suitable row permutations
we conclude that the absolute value of $d$ coincides with the absolute value
of the following determinant
\begin{equation*}
\begin{vmatrix}
y_{11} & \ldots & y_{1n} & x_{11} & \ldots & x_{1n} \\
\vdots & \ddots & \vdots & \vdots & \ddots & \vdots \\
y_{n1} & \ldots & y_{nn} & x_{n1} & \ldots & x_{nn} \\
0 & \ldots & 0 & 2x_{11} & \ldots & 2x_{1n} \\
0 & \ldots & 0 & \vdots & \ldots & \vdots \\
0 & \ldots & 0 & 2x_{n1} & \ldots & 2x_{nn}%
\end{vmatrix}%
\end{equation*}
which is nonzero by the linear independence of the sets $\Theta _{s}$ and
$\Theta _{c}$ respectively.
\end{proof}

\begin{theorem}
\label{main2}Let $S=\left(  s_{ij}\right)  $ and $C=\left( c_{ij}\right)  $
be matrices of order $n$ whose spectra (counted with their
multiplicities) are $\sigma (S)=\left(\lambda_{1},\lambda_{2},\ldots
,\lambda_{n}\right)  \ $and$\ \sigma (C)=\left(\mu_{1},\mu_{2},\ldots
,\mu_{n}\right)  $, respectively. Let $0\leq \gamma \leq 1$. If
\begin{equation}
\left\vert c_{ij}\right\vert \leq s_{ij}, \, 1\leq i,j\leq
n,\label{mayorize}
\end{equation}
(or equivalently if $S$, $S+C$ and $S-C$ are nonnegative matrices), then the matrices $\frac{1}{2}\left(  S+\gamma C\right)$ and $\frac{1}{2}\left(
S-\gamma C\right)$ are nonnegative and the nonnegative matrices
\begin{equation}
M_{\pm\gamma}=\left( M_{{ij}_{\pm\gamma}}\right)  ,\text{with \,} M_{{ij}_{\pm\gamma}}=
\begin{pmatrix}
\frac{s_{ij} \pm\gamma c_{ij}}{2} & \frac{s_{ij}\mp\gamma c_{ij}}{2}\\
\frac{s_{ij}\mp\gamma c_{ij}}{2} & \frac{s_{ij}\pm\gamma c_{ij}}{2}%
\end{pmatrix}
,\, 1\leq i,j\leq n\label{byblock}%
\end{equation}
realize, respectively, the following lists
\[
\sigma(S)\cup\gamma \sigma(C):=\left(  \lambda_{1},\lambda_{2},\ldots,\lambda_{n},\gamma\mu
_{1},\gamma \mu_{2},\ldots, \gamma \mu_{n}\right)
\] and
\[
\sigma(S)\cup\left(-\gamma\sigma(C)\right):=\left(  \lambda_{1},\lambda_{2},\ldots,\lambda_{n},-\gamma\mu
_{1},-\gamma\mu_{2},\ldots,-\gamma\mu_{n}\right).
\]

\end{theorem}

\begin{proof}
By the definitions of $\frac{1}{2}\left(  S+\gamma C\right)  \ $and$\ \frac{1}
{2}\left(  S-\gamma C\right)  \ $ and the condition in (\ref{mayorize})\ it is clear
that $M_{\pm\gamma}$ in (\ref{byblock}) are nonnegative matrices. By conditions
of Theorem \ref{main} one can see that each $\left(
i,j\right)  $-block of the matrix takes the form
$
\begin{pmatrix}
x_{ij} & y_{ij}\\
y_{ij} & x_{ij}%
\end{pmatrix}
$ and its spectrum is partitioned into the union of the spectra of the $n$-by-$n$ matrices
$\left(  x_{ij}+y_{ij}\right)  _{i,j=1}^{n}$ and $\left(  x_{ij}%
-y_{ij}\right)  _{i,j=1}^{n}$. If we impose that $S=\left(  x_{ij}%
+y_{ij}\right)  _{i,j=1}^{n}$ and $\pm\gamma C=\left(  x_{ij}-y_{ij}\right)
_{i,j=1}^{n}$ we obtain $s_{ij}=x_{ij}+y_{ij}$ and
$\pm\gamma c_{ij} =x_{ij}-y_{ij}.$ Thus, for both cases $x_{ij}=\frac{s_{ij} \pm \gamma c_{ij}}{2}$ and $y_{ij}
=\frac{s_{ij} \mp \gamma c_{ij}}{2}$,\ as it is required for the respective realization of the spectra
$\sigma(S) \cup \pm \gamma\sigma(C).$
\end{proof}

\begin{remark}
{\rm
\label{important}Note that in the previous result if $S=\left( s_{ij}\right) $
and $C=\left( c_{ij}\right) $, then
\begin{equation}
M_{\pm\gamma}=
\begin{pmatrix}
\frac{s_{11} \pm \gamma c_{11}}{2} & \frac{s_{11} \mp\gamma c_{11}}{2} & \ldots & \ldots &
\frac{s_{1n} \pm \gamma c_{1n}}{2} & \frac{s_{1n} \mp\gamma c_{1n}}{2} \\
\frac{s_{11} \mp\gamma c_{11}}{2} & \frac{s_{11} \pm \gamma c_{11}}{2} & \ldots & \ldots &
\frac{s_{1n}\mp\gamma c_{1n}}{2} & \frac{s_{1n}\pm \gamma c_{1n}}{2} \\
\vdots & \vdots & \ddots & \ddots & \vdots & \vdots \\
\vdots & \vdots & \ddots & \ddots & \vdots & \vdots \\
\frac{s_{n1} \pm \gamma c_{n1}}{2} & \frac{s_{n1}\mp \gamma c_{n1}}{2} & \ldots & \ldots &
\frac{s_{nn} \pm \gamma c_{nn}}{2} & \frac{s_{nn} \mp\gamma c_{nn}}{2} \\
\frac{s_{n1} \mp\gamma c_{n1}}{2} & \frac{s_{n1} \pm \gamma c_{n1}}{2} & \ldots & \ldots &
\frac{s_{nn}\mp\gamma c_{nn}}{2} & \frac{s_{nn}\pm \gamma c_{nn}}{2}%
\end{pmatrix}%
.  \label{aspectM}
\end{equation}
}
\end{remark}

The next corollary establishes the result when the matrices $S$ and $C$ are symmetric, both with prescribed list of eigenvalues.

\begin{corollary}
\label{syper}Let $S=\left(  s_{ij}\right)  $ and $C=\left(  c_{ij}\right)  $
be symmetric  matrices of orders $n$ whose spectra (counted with their
multiplicities) are $\sigma(S)=\left(  \lambda_{1},\lambda_{2},\ldots
,\lambda_{n}\right)  $ and $\sigma(C)=\left(  \mu_{1},\mu_{2},\ldots,\mu
_{n}\right)  $, respectively. Let $0\leq \gamma \leq 1$. Moreover, suppose that $\left\vert
c_{ij}\right\vert \leq s_{ij}$ for all $1\leq i,j\leq n.$ Then
$\frac{1}{2}\left(  S+\gamma C\right)  $ and $\frac{1}{2}\left(  S- \gamma C\right)  \ $ are
symmetric nonnegative matrices and
\[
M_{\pm \gamma }=\left(  M_{{ij}_{\pm \gamma }}\right) \text{ with }M_{{ij}_{\pm \gamma }}=%
\begin{pmatrix}
\frac{s_{ij}\pm \gamma  c_{ij}}{2} & \frac{s_{ij}\mp \gamma  c_{ij}}{2}\\
\frac{s_{ij}\mp \gamma  c_{ij}}{2} & \frac{s_{ij}\pm \gamma  c_{ij}}{2}%
\end{pmatrix},
\ \text{for} \ 1\leq i,j\leq n
\]
are symmetric nonnegative matrices such that, respectively, realize the following lists
\[
\sigma(S)\cup \left(\pm \gamma \sigma(C)\right)=\left(  \lambda_{1},\lambda_{2},\ldots,\lambda
_{n},\pm \gamma \mu_{1},\pm \gamma  \mu_{2},\ldots, \pm \gamma \mu_{n}\right).
\]

\end{corollary}

\begin{proof}
It is an immediate consequence of Theorem \ref{main2} that if the matrices $S$ and $C$  are symmetric, then the matrices $M_{\pm \gamma}$ obtained in
(\ref{aspectM}) are also symmetric.
\end{proof}

\begin{remark}
{\rm
\label{cond}We remark that for two permutatively equivalent $n$-by-$n$ matrices $S=\left(
s_{ij}\right)  $ and $C=\left(  c_{ij}\right)  $ whose first row,
are the an $n$-tuple $\left(  s_{1},\ldots,s_{n}\right),  $ and
$\left(  c_{1},\ldots,c_{n}\right)  ,$ respectively, the inequalities $\left\vert
c_{ij}\right\vert \leq s_{ij}$ hold if and only if $\left\vert
c_{i}\right\vert \leq s_{i}, 1\leq i\leq n.$
}
\end{remark}

\begin{theorem}
\label{main3}
Let $S=\left(  s_{ij}\right)  $ and $C=\left(  c_{ij}\right)  $ be permutatively equivalent matrices whose first row are the $n$-tuples
$\left(  s_{1},\ldots,s_{n}\right)  $ and $\left(  c_{1},\ldots,c_{n}\right),
$ respectively, such that $\left\vert c_{i}\right\vert \leq s_{i},1\leq i\leq n.$
Moreover, their spectra (counted with their multiplicities) are the lists
$\sigma(S)=\left(  \lambda_{1},\lambda_{2},\ldots,\lambda_{n}\right)  $ and
$\sigma(C)=\left(  \mu_{1},\mu_{2},\ldots,\mu_{n}\right)  $, respectively. Let $0\leq \gamma \leq 1$.
Then,  $\frac{1}{2}\left(  S+ \gamma C\right)  $ and $\frac{1}{2}\left(
S-\gamma C\right)  \ $ are nonnegative matrices, permutatively equivalent matrices and the
following matrices:

\begin{equation}
M_{\pm \gamma}=\left(  M_{{ij}_{\pm \gamma}}\right) \text{ with }M_{{ij}_{\pm \gamma}}=
\begin{pmatrix}
\frac{s_{ij}\pm \gamma c_{ij}}{2} & \frac{s_{ij}\mp \gamma c_{ij}}{2}\\
\frac{s_{ij}\mp \gamma c_{ij}}{2} & \frac{s_{ij}\pm \gamma c_{ij}}{2}%
\end{pmatrix},
\ 1\leq i,j\leq n\label{realizingM3}
\end{equation}
are permutative and realize, respectively, the following lists
\[
\sigma(S)\cup \left(\pm \gamma\sigma(C)\right)=\left(  \lambda_{1},\lambda_{2},\ldots,\lambda
_{n},\pm\gamma \mu_{1},\pm \gamma \mu_{2},\ldots, \pm \gamma  \mu_{n}\right).
\]
In particular, if the list $\left( \lambda_{1},\lambda
_{2},\ldots,\lambda_{n}\right)  $ is a Sule\u {\i}manova's type list and $\left( \mu_{1},\mu_{2},\ldots,\mu
_{n}\right) $ satisfies the condition
\begin{equation}
\mu_{1}+\mu_{2}+\cdots+\mu_{n}\leq\lambda_{1}+\lambda_{2}+\cdots+\lambda
_{n},\text{ }\label{cc1}%
\end{equation}
and%
\begin{equation}
\left\vert \frac{\mu_{1}+\mu_{2}+\cdots+\mu_{n}}{n}-\mu_{i}\right\vert
\leq\left (\frac{\lambda_{1}+\lambda_{2}+\cdots+\lambda_{n}}{n}-\lambda_{i} \right) \ \text{for} \ 2\leq i\leq n, \label{cc2}
\end{equation}
then, the lists $\sigma(S)\cup \left(\pm \gamma\sigma(C)\right)$ are respectively, realizable by the matrices $M_{\pm \gamma}$ in
(\ref{realizingM3}), where $S=\left(  s_{ij}\right) $ and $C=\left(c_{ij}\right)$ are the corresponding permutative matrices obtained from Theorem
\ref{pappa} and Lemma \ref{pappa copy(1)} by replacing with the lists of eigenvalues.
\end{theorem}

\begin{proof}
Is an immediate consequence of the fact that if the matrices $S$ and $C$ in Theorem
\ref{main2} are considered to be permutatively equivalent matrices then, by
Proposition \ref{cristi}, both $\frac{1}{2}\left(  S+\gamma C \right)  $ and $\frac
{1}{2}\left(  S- \gamma C \right)  $ are permutatively to $S.$ Therefore, by the shape of
the matrices in (\ref{aspectM}) the matrices $M_{\pm \gamma}$
in (\ref{realizingM3}) become permutative matrices.
In particular, if the matrices $S$ and $C$ and its spectra $\sigma(S)$ and
$\sigma(C)$, respectively, are as in the statement, by last statement of Remark \ref{important2} the
matrices $S$ and $\pm\gamma C$ that realize the spectra $\sigma(S)$ and
$\pm\gamma \sigma(C)$ are permutatively equivalent matrices. Then
$\frac{1}{2}\left(  S+ \gamma C\right)  $ and $\frac{1}{2}\left(  S- \gamma C\right)  \ $are
nonnegative permutative matrices, implying, by the above reasoning that the matrices $M_{\pm\gamma}$ in
(\ref{aspectM}), with the given description by (\ref{realizingM3}), will be also nonnegative permutative matrices. The conditions in (\ref{cc1}) and (\ref{cc2}) are derived from
the condition $\left\vert c_{i}\right\vert \leq s_{i}$, for all $1\leq i\leq
n$ when the $n$-tuples $\left(  s_{1},\ldots,s_{n}\right)  $ and $\left(
c_{1},\ldots,c_{n}\right)  $ are the first row of $S$ and $C$, respectively,
where the corresponding descriptions of $\left(  s_{1},\ldots,s_{n}\right)  $
and $\left(  c_{1},\ldots,c_{n}\right)  $ are obtained from Lemma \ref{pappa copy(1)} and Theorem
\ref{pappa} .
\end{proof}

Note that it is important that the matrices $S$ and $C$ are permutatively equivalent otherwise we can not guarantee that the matrices $M_{{\pm}\gamma}$ are permutative. In fact consider the following example:

\begin{example}
{\rm Both matrices $S$ and $C$ are not permutatively equivalent and $M$ constructed as in the previous theorem is not permutative.
\begin{eqnarray}
S &=&%
\begin{pmatrix}
1 & 2 & 3 \\
3 & 1 & 2 \\
2 & 3 & 1%
\end{pmatrix}%
, \\
\ C &=&%
\begin{pmatrix}
0 & 1 & 2 \\
2 & 0 & 1 \\
2 & 0 & 1%
\end{pmatrix}%
\end{eqnarray}%
\begin{equation}
M=%
\begin{pmatrix}
\frac{1}{2} & \frac{1}{2} & \frac{3}{2} & \frac{1}{2} & \frac{5}{2} & \frac{1%
}{2} \\
\frac{1}{2} & \frac{1}{2} & \frac{1}{2} & \frac{3}{2} & \frac{1}{2} & \frac{5%
}{2} \\
\frac{5}{2} & \frac{1}{2} & \frac{1}{2} & \frac{1}{2} & \frac{3}{2} & \frac{1%
}{2} \\
\frac{1}{2} & \frac{5}{2} & \frac{1}{2} & \frac{1}{2} & \frac{1}{2} & \frac{3%
}{2} \\
2 & 0 & \frac{3}{2} & \frac{3}{2} & 1 & 0 \\
0 & 2 & \frac{3}{2} & \frac{3}{2} & 0 & 1%
\end{pmatrix}%
\end{equation}
}
\end{example}

In the next examples $S$ and $C$ are permutatively equivalent.

\begin{example}
{\rm
Let $S$ and $C$ be the following circulant matrices, in consequence, they are
permutatively equivalent matrices
\begin{equation*}
S=%
\begin{pmatrix}
2 & 2 & 1 \\
1 & 2 & 2 \\
2 & 1 & 2%
\end{pmatrix}%
\text{ and }C=%
\begin{pmatrix}
0 & 0 & 1 \\
1 & 0 & 0 \\
0 & 1 & 0%
\end{pmatrix}%
\end{equation*}%
whose spectra, respectively, are the following lists
\begin{equation*}
\left(5,\frac{1+i\sqrt{3}}{2},\frac{1-i\sqrt{3}}{2}\right) \ \text{and}\
\left(1,\frac{1+i\sqrt{3}}{2},\frac{1-i\sqrt{3}}{2}\right).
\end{equation*}%
It is easy to see that the conditions of Theorem \ref{main3} are verified. In consequence, the $6$-by-$6$ matrix
\begin{equation*}
M=%
\begin{pmatrix}
1 & 1 & 1 & 1 & 1 & 0 \\
1 & 1 & 1 & 1 & 0 & 1 \\
1 & 0 & 1 & 1 & 1 & 1 \\
0 & 1 & 1 & 1 & 1 & 1 \\
1 & 1 & 1 & 0 & 1 & 1 \\
1 & 1 & 0 & 1 & 1 & 1%
\end{pmatrix}%
\end{equation*}%
is a permutative matrix and realizes the list
\begin{equation*}
\left( 5,\frac{1+i\sqrt{3}}{2},\frac{1-i\sqrt{3}}{2}, 1,%
\frac{1+i\sqrt{3}}{2},\frac{1-i\sqrt{3}}{2}\right) .
\end{equation*}
}
\end{example}

\begin{example}
{\rm
Let $S$ and $C$ be the following circulant matrices, in consequence, they are
permutatively equivalent matrices
\[
S=%
\begin{pmatrix}
2 & 2 & 1\\
1 & 2 & 2\\
2 & 1 & 2
\end{pmatrix}
\text{ and }C=%
\begin{pmatrix}
0 & 0 & -1\\
-1 & 0 & 0\\
0 & -1 & 0
\end{pmatrix}
\]
whose spectra, respectively, are
\[
\left( 5,\frac{1+i\sqrt{3}}{2},\frac{1-i\sqrt{3}}{2}\right)  \ \text{and}%
\ \left( -1,\frac{-1-i\sqrt{3}}{2},\frac{-1+i\sqrt{3}}{2}\right).
\]
It is easy to see that the conditions of Theorem \ref{main3} are verified. In consequence, the $6$-by-$6$ matrix
\[
M=%
\begin{pmatrix}
1 & 1 & 1 & 1 & 0 & 1\\
1 & 1 & 1 & 1 & 1 & 0\\
0 & 1 & 1 & 1 & 1 & 1\\
1 & 0 & 1 & 1 & 1 & 1\\
1 & 1 & 0 & 1 & 1 & 1\\
1 & 1 & 1 & 0 & 1 & 1
\end{pmatrix}
\]
is nonnegative permutative and realizes the spectrum
\[
\left( 5,\frac{1+i\sqrt{3}}{2},\frac{1-i\sqrt{3}}{2},
-1,\frac{-1-i\sqrt{3}}{2},\frac{-1+i\sqrt{3}}{2}\right)  .
\]
}
\end{example}

\begin{example}
{\rm
Let $\sigma=\left( 10,7,-3,-3,-2,-2,-2-1 \right)$. The following Sule\u{\i}manova sub-lists
$\left(7,-3,-2,-2\right)$ and $\left(  10,-3,-2,-1\right)$ can be obtained from $\sigma.$
Thus, the conditions of Theorem \ref{main3} hold and by Theorem \ref{pappa} the matrix that
realizes $\left(  10,-3,-2,-1\right)  $ is
\[
S=
\begin{pmatrix}
1 & 2 & 3 & 4\\
2 & 1 & 3 & 4\\
3 & 2 & 1 & 4\\
4 & 2 & 3 & 1
\end{pmatrix}
\]
and the matrix that realizes $\left(  7,-3,-2,-2\right)$ is
\[
C=
\begin{pmatrix}
0 & 2 & 2 & 3\\
2 & 0 & 2 & 3\\
2 & 2 & 0 & 3\\
3 & 2 & 2 & 0
\end{pmatrix}.
\]
Therefore, the matrix $M$ in (\ref{aspectM}) becomes
\[
M=%
\begin{pmatrix}
\frac{1}{2} & \frac{1}{2} & 2 & 0 & \frac{5}{2} & \frac{1}{2} & \frac{7}{2} &
\frac{1}{2}\\
\frac{1}{2} & \frac{1}{2} & 0 & 2 & \frac{1}{2} & \frac{5}{2} & \frac{1}{2} &
\frac{7}{2}\\
2 & 0 & \frac{1}{2} & \frac{1}{2} & \frac{5}{2} & \frac{1}{2} & \frac{7}{2} &
\frac{1}{2}\\
0 & 2 & \frac{1}{2} & \frac{1}{2} & \frac{1}{2} & \frac{5}{2} & \frac{1}{2} &
\frac{7}{2}\\
\frac{5}{2} & \frac{1}{2} & 2 & 0 & \frac{1}{2} & \frac{1}{2} & \frac{7}{2} &
\frac{1}{2}\\
\frac{1}{2} & \frac{5}{2} & 0 & 2 & \frac{1}{2} & \frac{1}{2} & \frac{1}{2} &
\frac{7}{2}\\
\frac{7}{2} & \frac{1}{2} & 2 & 0 & \frac{5}{2} & \frac{1}{2} & \frac{1}{2} &
\frac{1}{2}\\
\frac{1}{2} & \frac{7}{2} & 0 & 2 & \frac{1}{2} & \frac{5}{2} & \frac{1}{2} &
\frac{1}{2}%
\end{pmatrix}
\]
which is a permutative matrix and realizes the initial list.
}
\end{example}

\section{Real odd spectra}

We now present spectral results for matrices partitioned into blocks and with odd order. We start with the following spectral result that is presented in an algebraic closed field, $K$, for instance $K=\mathbb{C}.$

\begin{theorem}\label{main copy(1)}
Let $K$ be an algebraically closed field of characteristic $0$ and suppose
that $A=\left( A_{ij}\right) $ is an into block square matrix of order $2n+1$, where
\[
A_{ij}=\left\{
\begin{tabular}{cc}
$%
\begin{pmatrix}
a_{ij} & b_{ij} \\
b_{ij} & a_{ij}%
\end{pmatrix}%
$ & $1\leq i,j\leq n$ \\
$%
\begin{pmatrix}
a_{ij} \\
a_{ij}%
\end{pmatrix}%
$ & $1\leq i\leq n,\ j=n+1$ \\
$%
\begin{pmatrix}
a_{ij} & b_{ij}%
\end{pmatrix}%
$ & $i=n+1,$\ $1\leq j\leq n$ \\
$a_{ij}$ & $i=n+1,$\ $\ j=n+1.$%
\end{tabular}%
\right.
\]%
If
\[
s_{ij}=\left\{
\begin{tabular}{cc}
$a_{ij}+b_{ij}$ & $1\leq i,j\leq n$ \\
$a_{ij}$ & $1\leq i\leq n,\ j=n+1$ \\
$a_{ij}+b_{ij}$ & $i=n+1,$\ $1\leq j\leq n$ \\
$a_{ij}$ & $i=n+1,$\ $\ j=n+1$%
\end{tabular}%
\right.
\]%
and%
\[
c_{ij}=a_{ij}-b_{ij},\ 1\leq i,j\leq n.
\]%
Then
\[
\sigma \left( A\right) =\sigma \left( S\right) \cup \sigma \left( C\right),
\]%
where
\[
S=\left( s_{ij}\right) \text{ and }C=\left( c_{ij}\right) .
\]
\end{theorem}

\begin{proof}
Let $\left( \lambda ,v\right) $ be an eigenpair of $S$,\ with $v=\left(
v_{1},\ldots ,v_{n},v_{n+1}\right) ^{T}$, and consider the $\left(
2n+1\right) $-by-$1$ block vector $w=%
\begin{pmatrix}
\left( w_{j}\right)  \\
w_{n+1}%
\end{pmatrix}%
$, where by an abuse of notation, we have
\[
w_{j}=\left\{
\begin{tabular}{cc}
$v_{j}\mathbf{e}$ & $1\leq j\leq n$ \\
$v_{n+1}$ & $j=n+1.$%
\end{tabular}%
\right.
\]
Since
\begin{eqnarray*}
A_{ij}w_{j} &=&%
\begin{pmatrix}
a_{ij} & b_{ij} \\
b_{ij} & a_{ij}%
\end{pmatrix}%
v_{j}\mathbf{e=}%
\begin{pmatrix}
a_{ij}+b_{ij} \\
b_{ij}+a_{ij}%
\end{pmatrix}%
v_{j}=s_{ij}v_{j}\mathbf{e}\text{,\quad }1\leq i,j\leq n \\
A_{i,n+1}w_{n+1} &=&%
\begin{pmatrix}
a_{ij} \\
a_{ij}%
\end{pmatrix}%
w_{n+1}=s_{i,n+1}v_{n+1}\mathbf{e,}\text{\quad }1\leq i\leq n \\
A_{n+1,j}w_{j} &=&%
\begin{pmatrix}
a_{n+1,j} & b_{n+1,j}%
\end{pmatrix}%
w_{j}=%
\begin{pmatrix}
a_{n+1,j} & b_{n+1,j}%
\end{pmatrix}%
\left( v_{j}\mathbf{e}\right)\\
&=& \left( a_{n+1,j}+b_{n+1,j}\right)
v_{j}=s_{n+1,j}v_{j}\mathbf{,}\text{\quad }1\leq j\leq n
\end{eqnarray*}%
Finally,
\[
A_{n+1},_{n+1}w_{n+1}=s_{n+1,n+1}v_{n+1}.
\]%
Notice that, for every $i\in \{1,\ldots ,n\},$%
\begin{eqnarray*}
\sum\limits_{j=1}^{n+1}A_{ij}w_{j}& = &\sum
\limits_{j=1}^{n}A_{ij}w_{j}+A_{i,n+1}w_{n+1}\\
&= & \left(
\sum\limits_{j=1}^{n}s_{ij}v_{j}+s_{i,n+1}v_{n+1}\right) \mathbf{e=}\left(
\lambda v_{i}\right) \mathbf{e=}\lambda \left( v_{i}\mathbf{e}\right)
=\lambda w_{i}
\end{eqnarray*}
and%
\begin{eqnarray*}
\sum\limits_{j=1}^{n+1}A_{n+1,j}w_{j}& = & \sum \limits_{j=1}^{n}s_{n+1,j}v_{j}+A_{n+1,n+1}w_{n+1}\\
&= &\left(
\sum\limits_{j=1}^{n}s_{n+1,j}v_{j}+s_{n+1,n+1}v_{n+1}\right) \mathbf{=}%
\lambda v_{n+1}\mathbf{=}\lambda w_{n+1}
\end{eqnarray*}
i.e $\left( \lambda ,w\right) $ is an eigenpair of $A$.\ Thus $\sigma \left(
S\right) \subseteq \sigma \left( A\right) $.

Similarly, let $\left( \mu ,x\right) $ be an eigenpair of $C,$ with $%
x=\left( x_{1},\ldots ,x_{n}\right) ^{T}$ and consider the $\left(
2n+1\right) $-by-$1$ block vector $y=%
\begin{pmatrix}
\left( y_{j}\right)  \\
0%
\end{pmatrix}%
$, where $y_{j}:=x_{j}\mathbf{f}$ and $\mathbf{f=}\left( 1,-1\right) ^{T}$.
Since
\[
A_{ij}y_{j}=%
\begin{pmatrix}
a_{ij} & b_{ij} \\
b_{ij} & a_{ij}%
\end{pmatrix}%
x_{j}\mathbf{f=}%
\begin{pmatrix}
a_{ij}-b_{ij} \\
b_{ij}-a_{ij}%
\end{pmatrix}%
x_{j}=c_{ij}x_{j}\mathbf{f}\text{,}
\]%
notice that, for every $i=1,\ldots ,n,$%
\[
\sum\limits_{j=1}^{n+1}A_{ij}y_{j}=\left( \sum\limits_{j=1}^{n}c_{ij}x_{j}%
\mathbf{f+}A_{i,n+1}y_{n+1}\right) \mathbf{=}\left( \lambda x_{i}\right)
\mathbf{f=}\lambda \left( x_{i}\mathbf{f}\right) =\lambda y_{i}
\]%
i.e $\left( \mu ,y\right) $ be an eigenpair of $A$.\ Thus $\sigma \left(
C\right) \subseteq \sigma \left( A\right) $.
Suppose that
$$\Theta_{s}=\left\{  \left(  x_{1i}\mathbf{,}x_{2i}\mathbf{,\ldots,}x_{ni},x_{n+1,i}\right)^{T}:i=1,\ldots,n+1\right\}$$ and
$$\Theta_{c}=\left\{  \left(  y_{1i}\mathbf{,}y_{2i}\mathbf{,\ldots,}y_{ni}\right) ^{T}:i=1,\ldots,n\right\}$$ are bases of eigenvectors of $S$
and $C$, respectively. The result will follow after proving the linear independence of the following set $\Upsilon= \Upsilon_1 \cup \Upsilon_2$
where
$$
\small{
\Upsilon_1=\left\{ \left(  x_{1}\mathbf{e}_{2}^{T}, x_{2}\mathbf{e}_{2}^{T},\ldots,x_{n}\mathbf{e}_{2}^{T},x_{n+1}\right)
^{T}:\left(  x_{1},x_{2}\ldots,x_{n},x_{n+1}\right)
^{T}\in\text{ }\Theta_{s}\right\}}$$
and
$$
\small{
\Upsilon_2=
\left\{  \left(  y_{1}\mathbf{f}%
_{2}^{T},y_{2}\mathbf{f}_{2}^{T},\ldots,y_{n}\mathbf{f}%
_{2}^{T},0\right)  ^{T}:\left(  \eta,\left(  y_{1},y_{2}%
,\ldots,y_{n}\right)  ^{T}\right)  \text{ }\in\text{ }\Theta
_{c}\right\}}.
$$
To this aim, we study the next determinant:
\[
d=%
\begin{vmatrix}
y_{11} & \ldots & y_{1n} & x_{11} & \ldots & x_{1n} & x_{1,n+1}\\
-y_{11} & \ldots & -y_{1n} & x_{11} & \ldots & x_{1n} & x_{1,n+1}\\
\vdots & \ddots & \vdots & \vdots & \ddots & \vdots & \vdots\\
\vdots & \ddots & \vdots & \vdots & \ddots & \vdots & \vdots\\
y_{n1} & \ldots & y_{nn} & x_{n1} & \ldots & x_{nn} & x_{n,n+1}\\
-y_{n1} & \ldots & -y_{nn} & x_{n1} & \ldots & x_{nn} & x_{n,n+1}\\
0 & \ldots & 0 & x_{n+1,1} & \ldots & x_{n+1,n} & x_{n_{+1},n+1}%
\end{vmatrix}
.
\]
Note that $d$ stands for the determinant of a $(2n+1)$-by-$(2n+1)$ matrix obtained from the coordinates of the vectors in
$\Upsilon$.\ As before, adding rows and making suitable row permutations we
conclude that the absolute value of $d$ coincides with the absolute value of
the following determinant
\[%
\begin{vmatrix}
y_{11} & \ldots & y_{1n} & x_{11} & \ldots & x_{1n} & x_{1,n+1}\\
\vdots & \ddots & \vdots & \vdots & \ddots & \vdots & \vdots\\
y_{n1} & \ldots & y_{nn} & x_{n1} & \ldots & x_{nn} & x_{n,n+1}\\
0 & \ldots & 0 & 2x_{11} & \ldots & 2x_{1n} & 2x_{1,n+1}\\
0 & \ldots & 0 & \vdots & \ldots & \vdots & \vdots\\
0 & \ldots & 0 & 2x_{n,1} & \ldots & 2x_{n,n} & 2x_{n,n+1}\\
0 & \ldots & 0 & x_{n+1,1} & \ldots & x_{n+1,n} & x_{n_{+1},n+1}%
\end{vmatrix}
\]
which is nonzero by the linear independence of the set $\Theta_{s}$ and
$\Theta_{c}.\ $Thus the statement follows.

\end{proof}

\begin{theorem}
\label{main2 copy(1)}Let $S=\left(  s_{ij}\right)  $ be a matrix of order $n+1$ and $C=\left(  c_{ij}\right)$ a matrix of order $n$ whose spectra (counted with their multiplicities) are $\sigma(S)=\left(  \lambda_{1},\lambda_{2},\ldots,\lambda_{n+1}\right)  $ and
$\sigma(C)=\left(  \mu_{1},\mu_{2},\ldots,\mu_{n}\right)  $,
respectively. Moreover, suppose that $s_{ij}\geq\left\vert c_{ij} \right\vert $ for all $1\leq i,j \leq n, s_{i,n+1}\geq 0$, for $i=1,\ldots, n+1$ and $ \varphi_{n+1,i}^{\left( j \right)}\geq 0$, for $j=1,2$ and for $i=1,\ldots, n$. Then, for all
$0 \leq \gamma \leq 1, $
the nonnegative matrices
\begin{equation}
M_{\pm\gamma}=%
\begin{pmatrix}
\frac{s_{11}\pm\gamma c_{11}}{2} & \frac{s_{11}\mp\gamma c_{11}}{2} &
\ldots & \ldots & \frac{s_{1n}\pm\gamma c_{1n}}{2} & \frac{s_{1n}\mp\gamma
c_{1n}}{2} & s_{1,n+1}\\
\frac{s_{11}\mp\gamma c_{11}}{2} & \frac{s_{11}\pm\gamma c_{11}}{2} &
\ldots & \ldots & \frac{s_{1n}\mp\gamma c_{1n}}{2} & \frac{s_{1n}\pm\gamma
c_{1n}}{2} & s_{1,n+1}\\
\vdots & \vdots & \ddots & \ddots & \vdots & \vdots & \vdots\\
\vdots & \vdots & \ddots & \ddots & \vdots & \vdots & \vdots\\
\frac{s_{n1}\pm\gamma c_{n1}}{2} & \frac{s_{n1}\mp\gamma c_{n1}}{2} &
\ldots & \ldots & \frac{s_{nn}\pm\gamma c_{nn}}{2} & \frac{s_{nn}\mp\gamma
c_{nn}}{2} & s_{n,n+1}\\
\frac{s_{n1}\mp\gamma c_{n1}}{2} & \frac{s_{n1}\pm\gamma c_{n1}}{2} &
\ldots & \ldots & \frac{s_{nn}\mp\gamma c_{nn}}{2} & \frac{s_{nn}\pm\gamma
c_{nn}}{2} & s_{n,n+1}\\
\varphi_{n+1,1}^{\left(1\right)} & \varphi_{n+1,1}^{\left(2\right)} & \ldots & \ldots & \varphi
_{n+1,n}^{\left(1\right)} & \varphi_{n+1,n}^{\left(2\right)} & s_{n+1,n+1}%
\end{pmatrix}
,\label{matrixM2}%
\end{equation}
$$\varphi_{n+1,i}^{\left(1 \right)}+\varphi_{n+1,i}^{\left(2 \right)}=s_{n+1,i},\ {i=1,2,\ldots,n}$$
have spectra
\[
\sigma(S)\cup \pm\gamma\sigma(C)=\left(  \lambda_{1},\lambda_{2},\ldots,\lambda
_{n+1},\pm\gamma\mu_{1},\pm\gamma\mu_{2},\ldots,\pm\gamma\mu_{n}\right)  .
\]
\end{theorem}

\begin{proof}
The result follows from a direct application of Theorem \ref{main copy(1)} to the matrix $M$ in
(\ref{matrixM2}).
\end{proof}

\begin{example}
{\rm
Let  $S=
\begin{pmatrix}
1 & 1 & 0\\
1 & 2 & 1\\
0 & 1 & 1
\end{pmatrix}
$ and $C=%
\begin{pmatrix}
1 & 0\\
0 & 1
\end{pmatrix}
$ be matrices which spectrum are $\ \left(  1,0,3\right)  $ and $\left(
1,1\right)$, respectively.  Let
\[
M=
\begin{pmatrix}
1 & 0 & \frac{1}{2} & \frac{1}{2} & 0\\
0 & 1 & \frac{1}{2} & \frac{1}{2} & 0\\
\frac{1}{2} & \frac{1}{2} & \frac{3}{2} & \frac{1}{2} & 1\\
\frac{1}{2} & \frac{1}{2} & \frac{1}{2} & \frac{3}{2} & 1\\
0 & 0 & \frac{1}{2} & \frac{1}{2} & 1
\end{pmatrix}.
\]
Then, $M$ has eigenvalues
\[
\sigma(M)=\left(  3,0,1,1,1\right)  .
\]
}
\end{example}

\begin{theorem}
\label{odd1}Let $A=\left(  a_{i,j}\right)  $ and $B=\left(  b_{i,j}\right)$
be matrices of order $n$. Moreover, consider the $n$-tuples
$$\mathbf{x}=\left(  x_{1},\ldots,x_{n}\right)  ^{T}$$ and $$\mathbf{y}^{T}=\left(
y_{1},\ldots,y_{n}\right)  ^{T}.$$ Let
\begin{equation}
S=\left(
\begin{tabular}
[c]{c|c}%
$A+B$ & $\mathbf{x}$\\\hline
$2\mathbf{y}$ & $u$
\end{tabular}
\right)  \text{ and }C=A-B\label{odd}%
\end{equation}
with $A+B$ and $A-B$  nonnegative matrices and with $\mathbf{x}$, $\mathbf{y}$ and $u$  also nonnegative and, consider the matrix partitioned into blocks
\begin{equation}
M=
\begin{pmatrix}
M_{11} & M_{12} & \ldots & \ldots & M_{1n} & \mathbf{x}_{1}\\
M_{21} & M_{22} & \ldots & \ldots & M_{2n} & \mathbf{x}_{2}\\
\vdots & \vdots & \ddots & \ddots & \vdots & \vdots\\
\vdots & \vdots & \ddots & \ddots & \vdots & \vdots\\
M_{n1} & M_{n2} & \ldots & \ldots & M_{nn} & \mathbf{x}_{n}\\
\mathbf{y}_{1} & \mathbf{y}_{2} & \ldots & \ldots & \mathbf{y}_{n} & u
\end{pmatrix}
,\label{blockodd}
\end{equation}
where, for $1\leq i,j\leq n$
\[
M_{ij}=
\begin{pmatrix}
a_{ij} & b_{ij}\\
b_{ij} & a_{ij}%
\end{pmatrix}
\text{,\ }\mathbf{x}_{i}=\left(  x_{i},x_{i}\right)  ^{T}\text{ and
}\mathbf{y}_{i}=\left(  y_{i},y_{i}\right)  .
\]
Then
\[
\sigma(M)  =\sigma(S)  \cup\sigma(C).
\]
Moreover, the matrix $M$ is nonnegative symmetric when $A$, $B$
are symmetric matrices and $\ \mathbf{x}^{T}=\mathbf{y}.$
\end{theorem}

\begin{proof}
This result is a clear consequence of Theorem \ref{main copy(1)}.
\end{proof}

\begin{example}
{\rm
Let us consider the list $\sigma=\left(  2,\frac{-1+\sqrt{5}}{2},\frac
{-1+\sqrt{5}}{2},\frac{-1-\sqrt{5}}{2},\frac{-1-\sqrt{5}}{2}\right)  .$
}
\end{example}

If we want to apply the known sufficient conditions of Laffey and Smigoc, \cite{LaffSmgc}, it is
not possible to obtain a partition of $\sigma$ where each of its subset has
cardinality three. Nevertheless, the matrices
\[
S=%
\begin{pmatrix}
0 & 1 & 1\\
1 & 1 & 0\\
2 & 0 & 0
\end{pmatrix}
\text{ and }C=%
\begin{pmatrix}
0 & 1\\
1 & -1
\end{pmatrix}
\]
have respectively, the following list of eigenvalues
\[
\left(  2,\frac{-1+\sqrt{5}}{2},\frac{-1+\sqrt{5}}{2}\right)  \text{ and
}\left(  \frac{-1-\sqrt{5}}{2},\frac{-1-\sqrt{5}}{2}\right) .
\]
In consequence, we consider the matrix $M$ in (\ref{blockodd})
\[
M=%
\begin{pmatrix}
0 & 0 & 1 & 0 & 1\\
0 & 0 & 0 & 1 & 1\\
1 & 0 & 0 & 1 & 0\\
0 & 1 & 1 & 0 & 0\\
1 & 1 & 0 & 0 & 0
\end{pmatrix},
\]
and by Theorem \ref{odd1}
this matrix $M$ realizes the list
\[
\small{
\left(  2,\frac{-1+\sqrt{5}}{2},\frac{-1+\sqrt{5}}{2},\frac{-1-\sqrt{5}}%
{2},\frac{-1-\sqrt{5}}{2}\right)  .}
\]
In \cite{SmgcH} it was proven that if $\sigma=\left(  \lambda_{1},\ldots
,\lambda_{n}\right)  $ is a list of complex numbers whose Perron root is
$\lambda_{1}$ and with $\lambda_{i}\in\Upsilon=\left\{  z\in\mathbb{C}
\text{:}\operatorname{Re}z<0,\ \left\vert \sqrt{3}\operatorname{Re}
z\right\vert \geq\left\vert \operatorname{Im}z\right\vert \right\}  $, for
$i=2,\ldots,n$, then there exists a nonnegative matrix realizing the list
$\sigma$ if and only if $
{\displaystyle\sum\limits_{i=1}^{n}}
\lambda_{i}\geq0.$ The example below shows that the set $\Upsilon$ can widen out.

\begin{example}
{\rm
Let
\[
S=%
\begin{pmatrix}
4 & 3 & 5\\
5 & 4 & 3\\
3 & 5 & 4
\end{pmatrix}
\text{ and }C=%
\begin{pmatrix}
4 & 3\\
-3 & 4
\end{pmatrix}
\]
whose spectra are%
\begin{equation}
\sigma(S)=\left(  12,i\sqrt{3},-i\sqrt{3}\right)  \text{ and }\sigma(C)=\left(  4+3i,4-3i\right)  . \label{spectra}
\end{equation}
Both matrices satisfy the conditions of Theorem \ref{main2 copy(1)} and, the matrix $M$
obtained from $S$ and $C$ with the techniques above
\[
M=%
\begin{pmatrix}
4           &              0 &           3 &            0 & 5\\
0           &              4 &           0 &            3 & 5\\
1           &              4 &           4 &            0 & 3\\
4           &              1 &           0 &            4 & 3\\
\frac{3}{2} &    \frac{3}{2} & \frac{5}{2} & \frac{5}{2}  & 4
\end{pmatrix}
\]
realizes the complex list

\begin{equation*}
\sigma(M)=\left(  12,i\sqrt{3},-i\sqrt{3},4+3i,4-3i\right)  \\ \text{where},\ i\sqrt{3},-i\sqrt{3},4+3i,4-3i \notin{\Upsilon}.
\end{equation*}
}
\end{example}

With this example we illustrate the fact that it is possible to find a nonnegative matrix that realizes a certain list of complex numbers that are not only in $\Upsilon.$ Moreover, note that the list at the example also verifies the condition that the sum of its elements is greater or equal than zero.\\
The next example shows that accordingly to Theorem \ref{main2 copy(1)} the next matrix $M$ also realizes the complex list and, therefore it is worth to notice that there is more than one matrix that realizes it.
\begin{example}
{\rm
Let
\[
S=%
\begin{pmatrix}
4 & 3 & 5\\
5 & 4 & 3\\
3 & 5 & 4
\end{pmatrix}
\text{ and }C=%
\begin{pmatrix}
4 & 3\\
-3 & 4
\end{pmatrix}
\]
be the matrices as previous example,
whose spectra are as in (\ref{spectra}).
Both matrices satisfy the conditions of Theorem \ref{main2 copy(1)} and, the matrix $M$
obtained from $S$ and $C$ with the techniques above (recall that the construction of $M$ in (\ref{matrixM2}).
\[
M=%
\begin{pmatrix}
4 & 0 & 3 & 0 & 5\\
0 & 4 & 0 & 3 & 5\\
1 & 4 & 4 & 0 & 3\\
4 & 1 & 0 & 4 & 3\\
3 & 0 & 5 & 0 & 4
\end{pmatrix}
\]
realizes the same complex list.
}
\end{example}

\section{Guo Perturbations}

In what follows the lists are considered as ordered an $n$-tuples. Guo \cite{GuoWen}, in a partial continuation of a work by Fiedler extended some spectral properties of symmetric nonnegative matrices to general nonnegative matrices. Moreover, he introduced the following interesting question:

If the list $\sigma=\left(\lambda_{1},\lambda_{2},\ldots,\lambda_{n}\right)$ is symmetrically
realizable (that is, $\sigma$ is the spectrum of a symmetric nonnegative matrix), and $t>0$, whether (or not) the list $\sigma_{t}=\left(  \lambda_{1}+t,\lambda_{2}\pm t,\lambda_{3},\ldots,\lambda_{n}\right)  $ is also symmetrically realizable?

In \cite{RSGuo} the authors gave  an affirmative answer to this question in the case that the realizing matrix is circulant or left circulant.

They also presented  a necessary and sufficient condition for $\sigma$ to be the
spectrum of a nonnegative 
circulant matrix. The following result was presented.

\begin{theorem} \cite{RSGuo}
\label{guoper}Let $\sigma=\left(  \lambda_{1},\lambda_{2},\lambda_{3}%
,\ldots,\overline{\lambda}_{3},\overline{\lambda}_{2}\right)  $ be the
spectrum of an $n$-by-$n$ nonnegative circulant matrix. Let $t\geq0$ and
$\theta\in\mathbb{R}$. Then
\[
\sigma_{t}=\left(  \lambda_{1}+2t,\lambda_{2}\pm t\exp\left(  i\theta\right)
,\lambda_{3},\ldots,\overline{\lambda}_{3},\overline{\lambda}_{2}\pm
t\exp\left(  -i\theta\right)  \right)
\]
is also the spectrum of an $n$-by-$n$ nonnegative circulant matrix. Moreover,
if $n=2m+2\,,$ then
\end{theorem}%

\[
\sigma_{t}=\left(  \lambda_{1}+t,\lambda_{2},\lambda_{3},\ldots,\lambda
_{m+1},\lambda_{m+2}\pm t,\overline{\lambda}_{m+1},\ldots,\overline{\lambda
}_{3},\overline{\lambda}_{2}\right)
\]
is also the the spectrum of an $n$-by-$n$ nonnegative circulant matrix.

\begin{theorem}
Let $n=2m+2$ and consider the $n$-tuples $\sigma_{1}=\sigma(S)=\left(  \lambda_{1}%
,\lambda_{2},\lambda_{3},\ldots,\overline{\lambda}_{3},\overline{\lambda}%
_{2}\right)  $ and $\sigma_{2}=\sigma(C)=\left(  \beta_{1},\beta_{2},\beta_{3}%
,\ldots,\overline{\beta}_{3},\overline{\beta}_{2}\right)$ with, respectively, realizing
matrices $S$ and $C$ being ciculant matrices and such that the matrices $S,$ $S+C$ and
$S-C$ are nonnegative matrices (see necessary and sufficient conditions to
this fact, for instance, in \cite{RSGuo}).
Let $t_{1}$ and $t_{2}$ such that
\begin{equation}
t_{1}\geq\left\vert t_{2}\right\vert ,\label{may}
\end{equation}
then, there exists a nonnegative permutative matrix $M$ realizing the list
$\sigma_{s,t_{1}}\cup\sigma_{c,t_{2}}$, where
\[
\sigma_{s,t_{1}}=\left(  \lambda_{1}+t_{1},\lambda_{2},\lambda_{3}%
,\ldots,\lambda_{m+1},\lambda_{m+2}\pm t_{1},\overline{\lambda}_{m+1}%
,\ldots,\overline{\lambda}_{3},\overline{\lambda}_{2}\right)
\]
and
\[
\sigma_{c,t_{2}}=\left(  \beta_{1}+t_{2},\beta_{2},\beta_{3},\ldots
,\beta_{m+1},\beta_{m+2}\pm t_{2},\overline{\beta}_{m+1},\ldots,\overline
{\beta}_{3},\overline{\beta}_{2}\right)  .
\]
\end{theorem}

\begin{proof}
Let $r_{s}=\left(  s_{1},\ldots,s_{n}\right)  ^{T}$ and $r_{c}=\left(
c_{1},\ldots,c_{n}\right)  ^{T}$ be the first row of matrices $S$ and $C,$
respectively. In \cite{RSGuo}, it is shown that these rows satisfy
\[
r_{s}=\frac{1}{n}\overline{F}\sigma_{1}^{T}\text{ and }r_{c}=\frac{1}%
{n}\overline{F} \sigma_{2}^{T}%
\]
where $F$ is the $n$ by $n$ matrix,
\[
F=\left(  \omega^{\left(  k-1\right)  \left(  j-1\right)  }\right)  _{1\leq
k,j\leq n}\text{ and }\omega=\exp\left(  \tfrac{2\pi i}{n}\right)
\] and $\overline{F}$ is the matrix conjugate of $F$. Then, if $\widetilde{r}_{s}$ and $\widetilde
{r}_{c}$ are the first row of the realizing matrices of the spectra
$\sigma_{s,t_{1}}$ and $\sigma_{c,t_{2}}$ those rows satisfy
\begin{equation}
\widetilde{r}_{s}=\frac{1}{n}\overline{F}\sigma_{s,t_{1}}^{T}\text{ and
}\widetilde{r}_{c}=\frac{1}{n}\overline{F}\sigma_{c,t_{2}}^{T}.\label{eqq}%
\end{equation}
Let $\mathbf{e}_{1}$ and $\mathbf{e}_{m+2}$ be the first and the $\left(
m+2\right)$-nd canonical vectors of $\mathbb{C}^{n}$. Adding, at first,
and after taking difference on the expressions in (\ref{eqq}) we obtain%
\begin{align*}
\widetilde{r}_{s}+\widetilde{r}_{c}  & =\frac{1}{n}\overline{F}\left(
\sigma_{s,t_{1}}^{T}+\sigma_{c,t_{2}}^{T}\right)  \\
& = r_{s}+r_{c} +\left(  t_{1}+t_{2}\right)  \overline{F} \mathbf{e}_{1} \pm \left(
t_{1}+t_{2}\right)  \overline{F}  \mathbf{e}_{m+2}
\end{align*}
and
\[
\text{ }\widetilde{r}_{s}-\widetilde{r}_{c}= r_{s}-r_{c}   +\left(  t_{1}-t_{2}\right)
\overline{F} \mathbf{e}_{1} \pm\left(  t_{1}-t_{2}\right)  \overline{F} \mathbf{e}_{m+2}  .
\]

Since $S+C$ and $S-C$ are nonnegative then  $r_{s}+r_{c}$ and $r_{s}-r_{c}$ are nonnegative. Moreover both $t_{1}+t_{2}\geq 0,t_{1}-t_{2}\geq0\ $(due to (\ref{may})) by Theorem \ref{guoper}, therefore both $\widetilde{r}_{s}+\widetilde{r}_{c}$ and $\widetilde{r}_{s}-\widetilde{r}_{c}$ are nonnegative columns. In consequence the circulant
matrices $\widetilde{S}$, $\widetilde{S+C}$ and $\widetilde{S-C}, $ whose first rows, respectively, are $\widetilde{r}_{s},$
$\widetilde{r}_{s}+\widetilde{r}_{c}$ and $\widetilde{r}_{s}-\widetilde{r}_{c},$ are nonnegative matrices and by Theorem \ref{guoper} they are still circulant matrices and nonnegative. In consequence, using the techniques from the above section the
matrix $\widetilde{M}$ obtained from the circulant matrices $\widetilde{S+C}$
and $\widetilde{S-C}$ is a permutative circulant by blocks matrix
whose spectrum is $\sigma_{s,t_{1}}\cup\sigma_{c,t_{2}}$ as required.
\end{proof}

\textbf{Acknowledgments}.
\noindent The authors would like to thank the anonymous referee for his/her careful reading and for several
valuable comments which have improved the paper.\\
Enide Andrade was supported in part by the Portuguese Foundation for Science and Technology (FCT-Funda\c{c}\~{a}o para a Ci\^{e}ncia e a Tecnologia), through CIDMA - Center for Research and Development in Mathematics and Applications, within project UID/MAT/04106/2013. M. Robbiano was partially supported by project VRIDT UCN 170403003.

\end{document}